\documentclass{amsart}

\usepackage[utf8]{inputenc}
\usepackage{verbatim}
\usepackage{graphicx}
\usepackage{graphicx,caption2,psfrag,float,color}
\usepackage{amssymb}
\usepackage{amscd}
\usepackage{amsmath}

\newtheorem{theorem}{Theorem}[section]
\newtheorem{conj}[theorem]{Conjecture}

\newtheorem{lemma}[theorem]{Lemma}

\numberwithin{equation}{subsection}
\newtheorem{definition}[theorem]{Definition}

\title{
The rate of growth of moments of certain cotangent sums}
\author{Helmut Maier and Michael Th. Rassias}
\date{\today}
\address{Department of Mathematics, University of Ulm, Helmholtzstrasse 18, 8901 Ulm, Germany.}
\email{helmut.maier@uni-ulm.de}
\address{Department of Mathematics, ETH-Z\"{u}rich, R\"{a}mistrasse 101, 8092 Z\"{u}rich, Switzerland \& Department of Mathematics, Princeton University, Fine Hall, Washington Road, Princeton, NJ 08544-1000, USA}
\email{michail.rassias@math.ethz.ch, michailrassias@math.princeton.edu}
\thanks{}

\begin{document}

 \maketitle
 
\begin{abstract} We consider cotangent sums associated to the zeros of the Estermann zeta function considered by the authors 
in their previous paper \cite{mr}. We settle a question on the rate of growth of the moments of these cotangent sums left open in \cite{mr}, and
obtain a simpler proof of the equidistribution of these sums. 
\newline

\end{abstract}

\section{Introduction}
The authors in joint work and the second author in his thesis, investigated the distribution of cotangent sums
$$c_0\left(\frac{r}{b}\right)=-\sum_{m=1}^{b-1}\frac{m}{b}\cot\left(\frac{\pi m r}{b} \right),$$
as $r$ ranges over the set
$$\{r\::\: (r,b)=1,A_0b\leq r\leq A_1b  \},$$
where $A_0,A_1$ are fixed with $1/2<A_0<A_1<1$ and $b$ tends to infinity.\\
Especially, they considered the moments
$$H_k=\lim_{b\rightarrow+\infty}\phi(b)^{-1}b^{-2k}(A_1-A_0)^{-1}\sum_{\substack{A_0b\leq r\leq A_1b \\ (r,b)=1}}c_0\left(\frac{r}{b} \right)^{2k},\  k\in\mathbb{N},$$
where $\phi(\cdot)$ denotes the Euler phi-function.\\
They could show that all the moments $H_k$ exist and that
$$\lim_{k\rightarrow+\infty}H_k^{1/k}=+\infty$$
Thus the series $\sum_{k\geq 0}H_kx^{2k}$converges only for $x=0$.\\
It was left open, whether the series
\[
\sum_{k\geq 0}\frac{H_k}{(2k)!}x^k\tag{*}
\]
converges for values of $x$ different from $0$. This fact would considerably simplify the proof for
the distribution of the cotangent sums $c_0(r/b)$ (uniqueness of measures determined by their moments, see \cite{billi}, Section 30, The Method of Moments, Theorem 30.1).\\
Crucial for the investigation was the result:
$$H_k=\int_0^1\left(\frac{g(x)}{2\pi} \right)^{2k}dx,$$
where
$$g(x)=\sum_{l\geq 1}\frac{1-2\{lx\}}{l}.$$
The function $g$ has been also investigated in the paper \cite{bre} of R. de la Bret\`eche and G. Tenenbaum. Their ideas will be crucial in our paper. We shall show the following theorems.
\begin{theorem}\label{x:maint}
There exists a constant $C_0>0$, such that
$$\int_0^1|g(x)|^Ldx\leq C_0^LL^L,$$
for all $L\in\mathbb{N}$.
\end{theorem}
\begin{theorem}\label{x:122}
The series 
$$\sum_{k\geq 0}\frac{H_k}{(2k)!}x^k$$
diverges for $|x|>\pi^2$, where $x\in\mathbb{C}$.
\end{theorem}
From Theorem \ref{x:maint}, an \textit{affirmative answer} regarding the question of the positive radius of convergence of (*) follows. From Theorem \ref{x:122} it follows that the radius of convergence of the series (*) is finite.
\begin{conj}
The radius of convergence of the series (*) is $\pi^2$.
\end{conj}
\vspace{5mm}
\section{Continued fractions}
\begin{definition}
Let $\alpha\in[0,1]\setminus \mathbb{Q}$. Assume that
$$\alpha=[0;a_1,a_2,\ldots]=\frac{1}{a_1+\frac{1}{a_2+\frac{1}{a_3+\cdots}}} $$
is its continued fraction expansion with integers $a_i\geq 1$ for $i=1,2,\ldots$\\
We denote the partial quotients by $p_r/q_r$, i.e. 
$$[0;a_1,a_2,\ldots, a_r]=\frac{p_r}{q_r},\ \text{with}\ (p_r,q_r)=1.$$
We set $p_{-1}=1$, $q_{-1}=0$, $p_0=0$, $q_0=1$.
\end{definition}
\begin{definition}
The map $T\::\:(0,1)\rightarrow(0,1),$ $\alpha\mapsto\frac{1}{\alpha}-\left\lfloor\frac{1}{\alpha}\right\rfloor$ is called the continued fraction map
(or Gauss map).
\end{definition}
\begin{lemma}\label{x:quo}
The partial quotients $p_r$, $q_r$ satisfy the recursion:
\[
p_{r+1}=a_{r+1}p_r+p_{r-1}\ \ \text{and}\ \ q_{r+1}=a_{r+1}q_r+q_{r-1}.\tag{1}
\]
\end{lemma}
\noindent\textit{Proof.} (cf. \cite{hens}, p. 7).
\begin{lemma}\label{x:22}
For $\alpha=[0;a_1,a_2,\ldots,a_r,a_{r+1},\ldots]$, we have
\[
T^r\alpha=[0;a_{r+1},a_{r+2},\ldots ] \tag{2}
\]
The map $T$ preserves the measure
\[
\omega(\mathcal{E})=\frac{1}{\log 2}\int_{\mathcal{E}}\frac{dx}{1+x},\tag{3}
\]
i.e. $\omega(T(\mathcal{E}))=\omega(\mathcal{E})$, for all measurable sets $\mathcal{E}\subset (0,1)$.
\end{lemma}
\begin{proof}
The result (2) is well known and can be easily confirmed by direct computation. For (3) cf. \cite{harman}, p. 119.
\end{proof}
\begin{lemma}
There is a constant $A_0>1$, such that
$$q_r\geq A_0^r\:,$$
for all $r\in\mathbb{N}.$
\end{lemma}
\begin{proof}
This is well known and easily follows from (1) of Lemma \ref{x:quo}.
\end{proof}
\begin{definition}
Let $\alpha\in (0,1)\setminus \mathbb{Q}$, $r\in\mathbb{N}$. Then, we set
$$c(\alpha,r)=\sum_{j=0}^r\frac{\log q_{j+1}}{q_j},$$
\[
c(\alpha,+\infty)=\sum_{j=0}^{+\infty}\frac{\log q_{j+1}}{q_j}\in\mathbb{R}\cup\{+\infty\}
\]
We define the constant $c_0>0$, by
$$c_0\sum_{r\geq 0}A_0^{-r/2}=\frac{1}{4}$$
and define the sequence $(w^{(r)})$ by
$$w^{(r)}=\frac{1}{2}+c_0\sum_{j=0}^rA_0^{-j/2}.$$
For $z\in(0,+\infty)$, we define 
\begin{align*}
\mathcal{E}(z,0)&:=\{\alpha\in(0,1)\setminus \mathbb{Q}\::\:c(\alpha,1)\geq w^{(0)}z\},\ (w^{(0)}=1/2)\\
\mathcal{E}(z,r)&:=\{\alpha\in(0,1)\setminus \mathbb{Q}\::\:c(\alpha,r-1)< w^{(r-1)}z,\: c(\alpha,r)\geq w^{(r)}z\}\\
\mathcal{E}(z,+\infty)&:=\{\alpha\in(0,+\infty)\setminus \mathbb{Q}\::\:c(\alpha,+\infty)\geq z\}.
\end{align*}
\end{definition}
\begin{lemma}\label{x:44}
For $z\in(0,+\infty)$, it holds 
$$meas(\mathcal{E}(z,+\infty))\leq \sum_{r\geq 0}meas (\mathcal{E}(z,r)),$$
where $meas$ stands for the Lebesgue measure.
\end{lemma}
\begin{proof}
Assume that $\alpha\not\in \mathcal{E}(z,r)$, for every $r\in\mathbb{N}\cup\{0\}$. Then it follows by induction on $r$, that
$$c(\alpha,r)\leq w^{(r)}z$$
and thus
$$c(\alpha,+\infty)=\lim_{r\rightarrow+\infty}c(\alpha,r)\leq \frac{3}{4}z.$$
Therefore, if $\alpha\in \mathcal{E}(z,+\infty)$ we have $\alpha\in \mathcal{E}(z,r)$ for at least one value of $r\in\mathbb{N}\cup\{0\}$. Thus
$$\mathcal{E}(z,+\infty)\subset \bigcup_{r=0}^{+\infty}\mathcal{E}(z,r),$$
which proves Lemma \ref{x:44}
\end{proof}
\begin{lemma}\label{x:28}
There are absolute constants $z_0>0$ and $c_0>0$, such that
$$meas(\mathcal{E}(z,r))\leq \exp\left(-\frac{1}{2}c_0A_0^{r/2}z\right),$$
for all $z\geq z_0.$
\end{lemma}
\begin{proof}
Assume that $\alpha\in\mathcal{E}(z,r)$. We have
$$c(\alpha,r)=c(\alpha,r-1)+\frac{\log q_{r+1}}{q_r}.$$
The inequalities
$$c(\alpha,r-1)<w^{(r-1)}z\ \ \text{and}\ \ c(\alpha,r)\geq w^{(r)}z,$$ 
imply that
\[
\frac{\log q_{r+1}}{q_r}\geq \left(w^{(r)}-w^{(r-1)}  \right)z=c_0A_0^{-r/2}z\tag{4}
\]
and
\[
q_{r+1}\geq \exp\left(c_0A_0^{-r/2}q_rz \right)\geq \exp\left(c_0q_r^{1/2}z \right).\tag{5}
\]
From
$$q_{r+1}=a_{r+1}q_r+q_{r-1}\leq (a_{r+1}+1)q_r$$
we obtain
\begin{align*}
a_{r+1}&\geq q_{r+1}q_r^{-1}-1\geq \exp\left(c_0q_r^{1/2}z \right)q_r^{-1}-1\tag{6}\\
&\geq \exp\left(\frac{3}{4}c_0q_r^{1/2}z \right)\geq \exp\left(\frac{3}{4}c_0A_0^{r/2}z  \right), 
\end{align*}
if $z_0$ is sufficiently large.\\
We have for all $w>0$:
$$T^r\{\alpha=[0;a_1,\ldots,a_{r+1},\ldots], a_{r+1}\geq w\}=\{\alpha=[0;a_{r+1},\ldots],a_{r+1}\geq w\},$$
by Lemma \ref{x:22}.\\
Since $T$ preserves the measure $\omega$, we have:
\begin{align*}
\omega\{\alpha=[0;a_1,\ldots,a_{r+1},\ldots], a_{r+1}\geq w\}&=\omega\{\alpha=[0;a_{r+1},\ldots],a_{r+1}\geq w\}.
\end{align*}
Therefore
$$[0;a_{r+1},\ldots]\leq w^{-1}$$
and thus
\begin{align*}
\omega\{\alpha&=[0;a_1,\ldots,a_{r+1},\ldots ], a_{r+1}\geq w\}\tag{7} \\
&\leq \frac{1}{\log 2}\int_{0}^{w^{-1}}\frac{dx}{1+x}\leq 2w^{-1}.
\end{align*}
Applying (6) and (7) we obtain
\[
meas\left(\mathcal{E}(z,r)\right)\leq 2w^{-1}.\tag{8}
\] 
We set in (8):
$$w=\exp\left(\frac{3}{4}c_0A_0^{r/2}z\right).$$
Then
$$meas( \mathcal{E}(z,r))\leq \exp\left(-\frac{1}{2}c_0A_0^{r/2}z \right)  .$$
\end{proof}
\begin{lemma}\label{x:66}
There is a constant $c_1>0$, such that
$$meas(\mathcal{E}(z,+\infty))\leq \exp(-c_1z),\ \text{if}\ z\geq z_0.$$
\end{lemma}
\begin{proof}
This follows from Lemmas \ref{x:44} and \ref{x:28}.
\end{proof}
\vspace{5mm}
\section{Results of R. de la Bret\`eche and G. Tenenbaum}
R. de la Bret\`eche and G. Tenenbaum \cite{bre} prove the following result (Th\'eor\`eme 4.4):
\begin{theorem}
The function
$$g(\alpha)=\sum_{l\geq 1}\frac{1-2\{l\alpha\}}{l}$$
converges for $\alpha\in\mathbb{Q}$ if and only if 
$$\sum_{r\geq 1}(-1)^r\frac{\log q_{r+1}}{q_r}$$
converges. In this case
\[
g(\alpha)=-\sum_{m\geq 1}\frac{\tau(m)}{\pi m}\sin(2\pi m\alpha),\tag{**}
\]
where $\tau$ stands for the divisor function.
\end{theorem}
The following definitions are adopted from \cite{bre}, p. 8.
\begin{definition}
For a multiplicative function $g$ and $x,y$ with $1\leq y\leq x$ and $\theta\in\mathbb{R}$ we denote by
$$Z_g(x,y;\theta):=\sum_{n\in S(x,y)}g(n)\sin(2\pi\theta n), $$
where $$S(x,y)=\{n\leq x\::\:P(n)\leq y\},$$ 
$P(n)$ being the largest prime factor of $n$.\\
We set
$$\mu(\theta;Q):=\min_{1\leq m\leq Q}\left\| m\theta \right\|\leq \frac{1}{Q} $$
and 
$$q(\theta;Q):=\min\{q\::\:1\leq q\leq Q,\ \text{with}\ \|q\theta\|=\mu(\theta;Q)\}, $$
where $\|\cdot\|$ denotes the distance to the nearest integer.
\end{definition}
We have:
\begin{lemma}
Let $A>0$. For $x\geq 2$, 
$$Q_{x}:=\frac{x}{(\log x)^{4A+24}},$$
$$q:=q(\theta;Q_{x}),\ a\in\mathbb{Z},\ (a,q)=1,$$
$$|q\theta-a|\leq \frac{1}{Q_x},\ \theta_q:=\theta-\frac{a}{q},\ \theta\in\mathbb{R},  $$
one has uniformly
$$Z_\tau(x,x;\theta)=x(\log x)\left\{ \frac{\sin^2(\pi \theta_q x)}{\pi q\theta_q x} + O\left(\frac{(\log q)\log (1+(\theta_q x)^2)}{q|\theta_q|x\log x} \right)+\frac{1}{(\log x)^A}\right\}  $$
\end{lemma}
\begin{proof}
This is Lemma 11.2 of \cite{bre}, pp. 64-65.
\end{proof}
\begin{definition}
For $\theta\in\mathbb{R}\setminus\mathbb{Q}$ let $(q_m)_{m\geq 1}=(q_m(\theta))_{m\geq 1}$ denote the sequence of the denominators of 
the partial fractions of $\theta$. Let $a_m/q_m$ denote the $m$-th partial fraction of $\theta$.\\ 
We set
$$\varepsilon_m:=\theta-\frac{a_m}{q_m}. $$
The set of all real numbers for which $q(\theta;Q_x)=q_m$ is an interval defined by the conditions $q_m\leq Q_x<q_{m+1}$. We denote it
by $[\xi_m,\xi_{m+1}]$.
\end{definition}
Then, we have:
\begin{lemma}
For a positive real constant $B$, we have:
$$\xi_m \asymp q_m(\log q_{m})^B, $$
$$ |\varepsilon_m|\xi_m\asymp\frac{(\log q_m)^B}{q_{m+1}},  $$
$$ |\varepsilon_m|\xi_{m+1}\asymp\frac{(\log q_{m+1})^B}{q_{m}},  $$
where $K\asymp L$ denotes $K=O(L)$ and $L=O(K)$.
\end{lemma}
\begin{proof}
This is equation (6.3) of \cite{bre}, p. 22.
\end{proof}
\begin{lemma}\label{x:36}
Let $\alpha\in (0,1)\setminus \mathbb{Q}$. There are constants $c_2, c_3>0$, such that 
$$|g(\alpha)|\leq c_2c(\alpha,+\infty)+c_3.  $$
\end{lemma}
\begin{proof}
We closely follow \cite{bre}, p. 65. By partial summation, we obtain:
\begin{align*}
g(\alpha)&=\sum_{n\geq1}\frac{\tau(n)}{n}\sin\left( 2\pi n \alpha \right) =\int_1^{+\infty}Z_{\tau}(t,t;\alpha) \frac{dt}{t^2}\\
&=\sum_{m\geq 1}\left( \int_{\xi_m}^{\xi_{m+1}} Z_{\tau}(t,t;\alpha)\frac{dt}{t^2} \right).
\end{align*}
By equation 11.5 of \cite{bre}, p. 65, we have
$$\int_{\xi_m}^{\xi_{m+1}} Z_{\tau}(t,t;\alpha)\frac{dt}{t^2}=\frac{1}{2}\pi\: \text{sgn}(\varepsilon_m)\frac{\log q_{m+1}}{q_m}+O\left(\frac{1}{q_m^{1-1/B}}+\int_{\xi_m}^{\xi_{m+1}}\frac{dt}{t(\log t)^A} \right),$$
where $A$ is fixed, but arbitrarily large.\\
Therefore
\begin{align*}
g(\alpha)&=\int_1^{+\infty}Z_{\tau}(t,t;\alpha) \frac{dt}{t^2}\\
&\leq c_2\sum_{m\geq 1}\frac{\log q_{m+1}}{q_m}+\sum_{m\geq 1}q_m^{1-1/B}+\int_1^{+\infty}\frac{dt}{t(\log t)^A}\\
&\leq c_2\:c(\alpha,+\infty)+c_3,
\end{align*}
since the sequence $(q_m)_{m\geq 1}$ is growing exponentially and the integral converges if $A>1$. This completes the proof. 
\end{proof}
\textit{Proof of Theorem \ref{x:maint}.} Let $L\in\mathbb{N}$ and assume that $\alpha$ satisfies (**) (Th\'eor\`eme 4.4. of \cite{bre}) and $|g(\alpha)|\geq 4L.$\\
We apply Lemmas \ref{x:66} and \ref{x:36} and obtain
$$meas\{\alpha\::\:|g(\alpha)|\geq yL\}\leq \exp(-c_1yL).$$
Therefore 
\begin{align*}
\int_0^1|g(\alpha)|^Ld\alpha&\leq \sum_{j\geq0}\left((2^{j+1}L)^L\:meas\{\alpha\::\: 2^jL\leq|g(\alpha)|\leq 2^{j+1}L\}\right)\\
&\leq \sum_{j\geq 0}(2^{j+1}L)^L\exp(-c_12^jL)\leq C_0^LL^L.
\end{align*}
\qed\\
However, 
\begin{align*}
H_k&=\int_0^1\left(\frac{g(x)}{2\pi}\right)^{2k}dx=(2\pi)^{-2k}\int_0^1g(x)^{2k}dx\\
&\leq (2\pi)^{-2k}\:C_0^{2k}\:(2k)^{2k}\\
&=\left(\frac{C_0}{2\pi}\right)^{2k}(2k)^{2k},\\
\end{align*}
because of Theorem \ref{x:maint} with $L=2k$, $k\in\mathbb{N}$. \\
Also, 
$$(2k)^{2k}\leq (2k)!\: 3^{2k},  $$
for $k\geq k_0$, for some $k_0\in\mathbb{N}.$\\
Hence
\begin{align*}
\frac{H_k}{(2k)!}&\leq \left(\frac{C_0}{2\pi} \right)^{2k}3^{2k}  \\
&=\left(\frac{3\:C_0}{2\pi} \right)^{2k},
\end{align*}
for $k\geq k_0$, for some $k_0\in\mathbb{N}.$\\ 
Hence, the radius of convergence of the series 
$$\sum_{k\geq 0}\frac{H_k}{(2k)!}x^k  $$
is positive.\\ \\
For the proof of Theorem \ref{x:122}, the following definitions and lemmas will be used.
\begin{definition}\label{x:377}
For $k\in\mathbb{N}\cup\{0\}$ we set
$$I:=I(k)=\left[ 0,e^{-2k}\right]\ \ \text{and}\ \ l_0:=l_0(k)=e^{2k}.$$
We fix $\delta>0$ arbitrarily small and set
$$g_1(\alpha):=\sum_{l\leq l_0^{1-2\delta}}\frac{B(l\alpha)}{l},\ \ g_2(\alpha):=\sum_{l_0^{1-2\delta}<l\leq l_0^{1+2\delta}}\frac{B(l\alpha)}{l},\ \ g_3(\alpha):=\sum_{l> l_0^{1+2\delta}}\frac{B(l\alpha)}{l},$$
where $B(u)=1-2\{u\}$, $u\in\mathbb{R}$.
\end{definition}
In the sequel, we assume $k\geq k_0$, where $k_0\in\mathbb{N}$, sufficiently large.
\begin{lemma}\label{x:6l1}
We have $$g(\alpha)=g_1(\alpha)+g_2(\alpha)+g_3(\alpha),$$ for every $\alpha\in\mathbb{R}$.
\end{lemma}
\begin{proof}
It is obvious by the definition of $g(\alpha),\:g_1(\alpha),\:g_2(\alpha),\:g_3(\alpha)$.
\end{proof}
\begin{lemma}\label{x:6l2}
For $\alpha\in I$, we have 
$$g_1(\alpha)\geq(1-8\delta)2k,$$
for $k\in\mathbb{N}\cup\{0\}.$
\end{lemma}
\begin{proof}
For $\alpha\in I$, $l\leq l_0^{1-2\delta}$ we have $l\alpha\leq \delta$ and therefore 
$$B(l\alpha)\geq 1-4\delta$$
because of Definition \ref{x:377}. Thus
$$g_1(\alpha)\geq (1-4\delta)\sum_{l\leq l_0^{1-2\delta}}\frac{1}{l}\:.$$
From the formula 
$$\sum_{m\leq u}\frac{1}{m}=\log u+O(1)\ \ (u\rightarrow+\infty) ,$$
we have
$$g_1(\alpha)\geq  (1-8\delta)2k. $$
\end{proof}
\begin{lemma}\label{x:6l3}
It holds
$$|g_2(\alpha)|\leq 16\delta k,$$
for $k\in\mathbb{N}\cup\{0\}$ and sufficiently small $\delta>0$.
\end{lemma}
\begin{proof}
We have
\begin{align*}
|g_2(\alpha)|\leq\sum_{l_0^{1-2\delta}<l\leq l_0^{1+2\delta}}\frac{1}{l}&\leq 2\left( \log(l_0^{1+2\delta})-\log(l_0^{1-2\delta}) \right)\\
&\leq 16\delta k.
\end{align*}
\end{proof}
\begin{lemma}\label{x:6l4}
For all $\alpha\in I$ that do not belong to an exceptional set $\mathcal{E}$ with measure
$$\text{meas}(\mathcal{E})\leq e^{-2k(1+\delta)},$$
we have
$$|g_3(\alpha)|\leq \delta k.$$
\end{lemma}
\begin{proof}
The function $g_3$ has the Fourier expansion:
$$g_3(\alpha)=\sum_{l>l_0^{1+2\delta}}c(l)\:e(l\alpha),$$
where $c(l)=O(l^{-1+\epsilon})$ for $\epsilon$ arbitrarily small, by Lemma 5.6 of \cite{mr}.\\
\noindent By Parseval's identity we have
$$\int_0^1g_3(\alpha)^2d\alpha=\sum_{l>l_0^{1+2\delta}}c(l)^2=O\left( \sum_{l>l_0^{1+2\delta}}l^{-2+2\epsilon}\right)=O\left(l_0^{-1-3\delta/2} \right).$$
Let 
$$\mathcal{E}=\{\alpha\::\: |g_3(\alpha)|>\delta k\}.$$
Then
\begin{align*}
(meas(\mathcal{E}))(\delta k)^2&\leq \int_{\mathcal{E}}g_3(\alpha)^2d\alpha\leq \int_0^1g_3(\alpha)^2d\alpha  \\
&=O\left( l_0^{-1-3\delta/2} \right).
\end{align*}
Therefore
$$meas(\mathcal{E})\leq O\left((\delta k)^{-2}\:l_0^{-1-3\delta/2} \right)=O\left(e^{-2k(1+\delta)}\right).$$
This completes the proof of the Lemma.
\end{proof}
\noindent \textit{Proof of Theorem \ref{x:122}.} \\
\noindent By Lemmas \ref{x:6l2}, \ref{x:6l3} and \ref{x:6l4}, we have
$$|g(\alpha)|\geq |g_1(\alpha)|- |g_2(\alpha)|- |g_3(\alpha)| \geq (1-20\delta)2k,$$
for all $\alpha\in I$ except for those values of $\alpha$ that belong to an exceptional set $$\mathcal{E}(I):=\mathcal{E}\cap I\subset I$$ with
$$\text{meas}(\mathcal{E}(I))\leq \frac{1}{2}|I|,$$
where $|I|$ stands for the length of $I$. Hence, we obtain
\begin{align*}
H_{k}=\int_0^1\left(\frac{g(\alpha)}{\pi}\right)^{2k}d\alpha&\geq \frac{1}{2}|I|\left(\frac{1-20\delta}{\pi}\:2k  \right)^{2k}\\
&=\frac{1}{2}e^{-2k}\exp(2k\log 2k)\left(\frac{1-20\delta}{\pi} \right)^{2k}.
\end{align*}
By Stirling's formula we have
$$(2k)!\geq \exp(2k\log 2k)\:\exp(-(1-\delta)2k)$$
and therefore
$$\frac{H_k}{(2k)!}\geq \frac{1}{2}e^{2\delta k}\left(\frac{1-20\delta}{\pi}\right)^{2k}.$$
Since $\delta>0$ can be fixed arbitrarily small, we have
$$\limsup_{k\rightarrow+\infty}\left(\frac{H_k}{(2k)!}\right)^{1/k} \geq \frac{1}{\pi^2}.$$
Therefore, the series
$$\sum_{k\geq 0}\frac{H_k}{(2k)!}x^{k}$$
diverges for $|x|>\pi^2$, where $x\in\mathbb{C}$. This completes the proof of Theorem \ref{x:122}.
\begin{flushright}
\qedsymbol
\end{flushright}

%
%
%
%
\section{The distribution of the cotangent sums $c_0\left(\frac{r}{b} \right)$}
We now give a simpler proof of Theorem 5.2 of \cite{mr} regarding the equidistribution of $c_0(r/b)$ for fixed large positive integer values
of $b$ and $A_0b\leq r\leq A_1b$, where $1/2<A_0<A_1<1$. We need the following Lemmas and Definitions from \cite{billi}.
\begin{lemma}\label{x:441}
Let $\mu$ be a probability measure on the line having finite moments
$$\alpha_k=\int_{-\infty}^{+\infty}x^k\mu(dx)$$
of all orders. If the power series 
$$\sum_{k\geq 1}\alpha_k\frac{r^k}{k!}  $$
has a positive radius of convergence, then $\mu$ is the only probability measure with the moments $\alpha_1,\alpha_2,\ldots$
\end{lemma}
\begin{proof}
This is Theorem 30.1 of \cite{billi}, pp. 388-389.
\end{proof}
\begin{definition}\label{x:44442}
A probability measure satisfying the conclusion of Lemma \ref{x:441} is said to be \textbf{determined by its moments}.
\end{definition}
\begin{definition}\label{x:433}
A sequence $(F_n)_{n\geq 1}$ of distribution functions is said to \textbf{converge weakly} to the distribution function $F$ (denoted by 
$F_n\Rightarrow F$) if 
$$\lim_{n\rightarrow +\infty}F_n(x)=F(x)$$
for every point $x$ of continuity of $F(x)$.\\
A sequence $(X_n)_{n\geq 1}$ of random variables is said to \textbf{converge in distribution} (or \textbf{in law}) towards a random variable $X$ 
(denoted by $X_n\Rightarrow X$) with distribution function $F$, if and only if $F_n\Rightarrow F$, that is 
$X_n\Rightarrow X\Leftrightarrow F_n\Rightarrow F$.
\end{definition}
\begin{lemma}
For a sequence $(X_n)_{n\geq 1}$ of random variables and a random variable $X$, we have $X_n\Rightarrow X$ if and only if
$$\lim_{n\rightarrow +\infty}P[X_n\leq x]=P[X\leq x]$$
for every $x\in\mathbb{R}$, such that $P[X=x]=0$.
\end{lemma}
\begin{proof}
This follows immediately from Definition \ref{x:433}.
\end{proof}
\begin{lemma}\label{x:4455}
Suppose that the distribution of $X$ is determined by its moments and that the $X_n$ have moments of all orders, as well as
$$ \lim_{n\rightarrow +\infty}E(X_n^{r})=E(X^{r})$$
for $r=1,2,3,\ldots$\:. Then $X_n\Rightarrow X$.
\end{lemma}
\begin{proof}
This is Theorem 30.2 of \cite{billi}, p. 390.
\end{proof}
We now recall the Definition 5.1 and Theorem 5.2 form \cite{mr}.
\begin{definition}
For $z\in\mathbb{R}$, let 
$$F(z)=\text{meas}\{\alpha\in[0,1]\::\: g(\alpha)\leq z  \},$$
where $``\text{meas}"$ denotes the Lebesgue measure,
$$g(\alpha)=\sum_{l=1}^{+\infty}\frac{1-2\{l\alpha\}}{l}$$  
and 
$$C_0(\mathbb{R})=\{f\in C(\mathbb{R})\::\: \forall\: \epsilon>0,\: \exists\:\text{a compact set}\ \mathcal{K}\subset\mathbb{R},\:\text{such that}\ |f(x)|<\epsilon,\forall\: x\not\in \mathcal{K}  \}.$$
\end{definition}

\begin{theorem}\label{x:introprob}
i) $F$ is a continuous function of $z$.\\
ii) Let $A_0,$ $A_1$ be fixed constants, such that $1/2< A_0<A_1<1$. Let also
$$H_k=\int_0^1\left(\frac{g(x)}{\pi}\right)^{2k}dx,$$
where $H_k$ is a positive constant depending only on $k$, $k\in\mathbb{N}$.\\ 
There is a unique positive measure $\mu$ on $\mathbb{R}$ with the following properties:\\
\ (a) For $\alpha<\beta\in\mathbb{R}$ we have
$$\mu([\alpha,\beta])=(A_1-A_0)(F(\beta)-F(\alpha)).$$
(b)
\begin{equation}
\int x^kd\mu=\left\{
\begin{array}{l l}
    (A_1-A_0)H_{k/2}\:, & \quad \text{for even}\: k\\
    0\:, & \quad \text{otherwise}\:.\\
  \end{array} \right.
 \nonumber
\end{equation}
(c) For all $f\in C_0(\mathbb{R})$, we have
$$\lim_{b\rightarrow+\infty}\frac{1}{\phi(b)}\sum_{\substack{r\::\: (r,b)=1\\ A_0b\leq r\leq A_1b}}f\left( \frac{1}{b}c_0\left( \frac{r}{b}\right) \right)=\int f\:d\mu,$$
where $\phi(\cdot)$ denotes the Euler phi-function.
\end{theorem}
We now state and give a new proof of a special case of Theorem \ref{x:introprob} (c) from which the complete Theorem \ref{x:introprob} follows by the definition of the abstract Lebesgue integral.
\begin{theorem}\label{x:48888}
Let $A_0$, $A_1$ be fixed constants, such that $1/2<A_0<A_1<1$, then we have for $\alpha<\beta\in\mathbb{R}:$
\begin{align*}
&\lim_{b\rightarrow+\infty}\frac{1}{\phi(b)}\left|\left\{ r\::\: (r,b)=1,\: A_0b\leq r\leq A_1b,\: \alpha b\leq c_0\left(\frac{r}{b}\right) \leq \beta b\right\}\right|\\ 
&=(A_1-A_0)(F(\beta)-F(\alpha)).
\end{align*}
\end{theorem}
\begin{proof}
Let $(b_n)_{n\geq 1}$ be a sequence of positive integers with $b_n\rightarrow+\infty$ as $n\rightarrow+\infty$. We set
$$X_n=\frac{1}{b_n}c_0\left(\frac{r}{b_n} \right) $$
and consider $X_n$ as a random variable on the probability space 
$$ \Omega_n=\{ r\::\: (r,b_n)=1, A_0b_n\leq r\leq A_1b_n\} $$
with the counting measure 
$$\mu_n(\mathcal{E})=\frac{|\mathcal{E}|}{|\Omega_n|} $$
for all $\mathcal{E}\subset \Omega_n$.\\
By Lemma 5.13 of \cite{mr}, we have
$$ \lim_{n\rightarrow+\infty}\mu_n([\alpha,\beta])=(A_1-A_0)(F(\beta)-F(\alpha))$$
for all $\alpha<\beta\in\mathbb{R}$.\\
By Theorem \ref{x:maint}, Lemma \ref{x:441} and Definition \ref{x:44442} the measure $\mu$ given by 
$$\mu([\alpha,\beta])=(A_1-A_0)(F(\beta)-F(\alpha)),$$
is determined by its moments. By Theorem \ref{x:introprob} we have
$$\lim_{n\rightarrow+\infty}E(X_n^r)=E(X^r).  $$
Thus, Lemma \ref{x:4455} implies $X_n\Rightarrow X$, where $X=g(\alpha)$ is a random variable on the
probability space $[0,1]$. Since $F$ is a continuous function by Theorem 5.2(i) of \cite{mr}, the claim of Theorem \ref{x:48888} follows.
\end{proof}
\noindent\textbf{Acknowledgments.} The second author (M. Th. Rassias) expresses his gratitude to Professor E. Kowalski, who proposed to him this inspiring area of research, for providing constructive guidance and for granting him support for Postdoctoral research.


\vspace{10mm}
\begin{thebibliography}{99}%
\bibitem{billi} P. Billingsley, \textit{Probability and Measure}, John Wiley, New York, 1995.
\bibitem{bre} R. de la Bret\`eche and G. Tenenbaum, \textit{S\'eries trigonom\'etriques \`a coefficients arithm\'etiques}, J.  Anal. Math., 92(2004), 1--79.
\bibitem{harman} G. Harman, \textit{Metric Number Theory}, Oxford Univ. Press, Oxford, New York, 1998.
\bibitem{hens} D. Hensley, \textit{Continued Fractions}, World Scientific Publ. Co., Singapore, 2006.
\bibitem{mr} H. Maier and M. Th. Rassias, \textit{Generalizations of a cotangent sum associated to the Estermann zeta function}, preprint.

\end{thebibliography}
\end{document}